\documentclass[12pt, reqno]{amsart}
\usepackage{graphicx}
\usepackage{subfigure}
\usepackage{eepic}
\usepackage{xcolor}
\selectcolormodel{gray}
\usepackage{tikz}
\usepackage{amsmath}
\usepackage{a4wide}
\usepackage[utf8]{inputenc}
\usepackage{amssymb}
\usepackage{amsopn}
\usepackage{epsfig}
\usepackage{amsfonts}
\usepackage{latexsym}
\usepackage{amsthm}
\usepackage{enumerate}
\usepackage[UKenglish]{babel}
\usepackage{verbatim}
\usepackage{color}

\newtheorem{thm}{Theorem}[section]
\newtheorem{lma}[thm]{Lemma}
\newtheorem{ass}[thm]{Assumption}
\newtheorem{cor}[thm]{Corollary}

\newtheorem{prop}[thm]{Proposition}

\newtheorem{rem}[thm]{Remark}

\newcommand{\R}{\mathbb{R}}

\newcommand{\N}{\mathbb{N}}

\newcommand{\C}{\mathbb{C}}

\providecommand{\norm}[1]{\lVert#1\rVert}

\renewcommand{\geq}{\geqslant}
\renewcommand{\leq}{\leqslant}
\renewcommand{\epsilon}{\varepsilon}

\renewcommand{\P}{\mathbb{P}}

\renewcommand{\i}{\mathbf{i}}

\newcommand{\hd}{\dim_\mathrm{H}}

\newcommand{\bd}{\dim_\mathrm{B}}

\renewcommand{\geq}{\geqslant}
\renewcommand{\leq}{\leqslant}

\providecommand{\A}{\mathcal{A}}
\renewcommand{\i}{\mathbf{i}}
\providecommand{\I}{\mathcal{I}}

\renewcommand{\u}{\mathbf{u}}

\renewcommand{\l}{\mathcal{L}}
\providecommand{\x}{\mathbf{x}}

\renewcommand{\t}{\mathbf{t}}

\providecommand{\i}{\mathbf{i}}
\providecommand{\M}{\mathcal{M}}
\providecommand{\norm}[1]{\lVert#1\rVert}
\providecommand{\bignorm}[1]{\left\lVert#1\right\rVert}
\providecommand{\one}{\mathbf{1}}

\providecommand{\myfloor}[1]{\left \lfloor #1 \right \rfloor }
\renewcommand{\c}{\mathcal{C}}

\newcommand{\interior}[1]{%
  {\kern0pt#1}^{\mathrm{o}}%
}

\begin{document}

\title[]{Analyticity of the affinity dimension \\ for planar iterated function systems with matrices which preserve a cone}

\author{Natalia Jurga and Ian D. Morris}
\address{Natalia Jurga: Department of Mathematics, University of Surrey, Guildford, GU2 7XH, UK}
\email{N.Jurga@surrey.ac.uk}
\address{Ian Morris: Department of Mathematics, University of Surrey, Guildford, GU2 7XH, UK}
\email{I.Morris@surrey.ac.uk }

\date{\today}

\subjclass[2010]{}

\begin{abstract}
The sub-additive pressure function $P(s)$ for an affine iterated function system (IFS) and the affinity dimension, defined as the unique solution $s_0$ to $P(s_0)=1$, were introduced by K. Falconer in his seminal 1988 paper on self-affine fractals. The affinity dimension prescribes a value for the Hausdorff dimension of a self-affine set which is known to be correct in generic cases and in an increasing range of explicit cases. It was shown by Feng and Shmerkin in 2014 that the affinity dimension depends continuously on the IFS. In this article we prove that when the linear parts of the affinities which define the IFS are $2 \times 2$ matrices which strictly preserve a common cone, the sub-additive pressure is locally real analytic as a function of the matrix coefficients of the linear parts of the affinities. In this setting we also show that the sub-additive pressure is piecewise real analytic in $s$, implying that the affinity dimension is locally analytic in the matrix coefficients.  Combining this with a recent result of B\'ar\'any, Hochman and Rapaport we obtain results concerning the analyticity of the Hausdorff dimension for certain families of planar self-affine sets.

\end{abstract}

\maketitle

\section{Introduction}\label{intro}

Let $\Phi=\{S^{(i)}: \R^d \to \R^d: i \in \I\}$ be a finite collection of contracting affine maps, that is, $S^{(i)}(\cdot)= A^{(i)} (\cdot)+b^{(i)}$ where $A^{(i)}$ is an invertible contracting $d \times d$ matrix and $b^{(i)} \in \R^d$ is a translation vector. It is well-known that there exists a unique non-empty compact set $F \subseteq \R^d$ satisfying:
\[
F \ = \ \bigcup_{i \in \mathcal{I}} S^{(i)}(F)
\]
which is known as the attractor of the iterated function system (IFS) $\{S^{(i)}\}_{i \in \mathcal{I}}$ and is called a \emph{self-affine set}. In the special case that all of the maps are similarities  we say that $F$ is a \emph{self-similar set}. A large part of the dimension theory of self-similar sets is well understood. For example, if we denote the contraction ratios of the similarities $S^{(i)}$ by $\mathbf{r}:=\{r_i\}_{i \in \I}$ then under suitable separation assumptions on the pieces $\{S^{(i)}(F)\}_{i \in \I}$ it is well known that all notions of dimension of $F$ coincide and the common value is given by the solution $s$ to the pressure-type formula
\begin{eqnarray}
P_{\mathbf{r}}(s)=\sum_{i \in \I} r_i^s=1
\label{hm}
\end{eqnarray}
which we call the \emph{similarity dimension}. However, when we pass to the more general self-affine setting where the matrices $A^{(i)}$ are allowed to exhibit different rates of contraction in different directions, the problem of calculating the dimension becomes drastically more complex.

For any invertible  $d \times d$ matrix $A$ write $\alpha_1(A) \geq \alpha_2(A) \geq \cdots \geq \alpha_d(A) >0$ for the singular values of $A$. Following \cite{falconer}, let 
$$\phi^s(A):= \alpha_1(A)\alpha_2(A) \ldots \alpha_{\myfloor{s}}(A)\alpha_{\myfloor{s}+1}(A)^{s-\myfloor{s}}$$
when $0 \leq s \leq d$, and $\phi^s(A):=|\det A|^{s/d}$ when $s\geq d$. 
We call $\phi^s(A)$ the \emph{singular value function}. Let $\A=\{A^{(i)}: i \in \I\}$ denote the set of linear parts of the affine maps $S^{(i)}$, so that $\A$ is a set of contracting invertible $d \times d$ matrices, and denote $\A^n=\{A^{(i_1)} \cdots A^{(i_n)}: i_j \in \I\}$. The singular value function is submultiplicative in the sense that $\phi^s(AB) \leq \phi^s(A)\phi^s(B)$ for any $A, B \in \bigcup_{n \in \N} \A^n$. Therefore, the sub-additive pressure can be defined as
$$P_{\A}(s)=\lim_{n\to\infty}\left(\sum_{A \in \A^n}\phi^s(A )\right)^{\frac{1}{n}}.$$
Falconer \cite{falconer} introduced the \emph{affinity dimension} of $F$ which is given by the unique value $s_0>0$ such that $P_{\A}(s_0)=1$. Since $s_0$ only depends on the set of matrices $\A$ (and not on the translation vectors) we will denote the affinity dimension of $F$ by $\dim \A$. Falconer showed that the Hausdorff dimension of a self-affine set is `typically' given by the affinity dimension, and moreover that the affinity dimension is always an upper bound for the Hausdorff dimension of a self-affine set. Falconer's proof did not construct \emph{explicit} examples of self-affine sets with affinity dimension equal to the Hausdorff dimension, but such examples have been constructed in a range of subsequent articles such as \cite{hl}, \cite{ks},\cite{ms}, all within the planar setting. The most general result of this kind to date is due to B\'ar\'any, Hochman and Rapaport \cite{bhr} which we briefly describe below. We let $\hd$ and $\bd$ respectively denote the Hausdorff dimension and box dimension of a subset of $\mathbb{R}^d$.

\begin{thm}[Theorem 1.1 \cite{bhr}]
Let $\Phi=\{S^{(i)}: \R^2\to\R^2 : i \in \I\}$ be an affine iterated function system and $F= \bigcup_{i \in \I} S^{(i)}(F)$ be the associated self-affine set. Suppose that
\begin{enumerate}
\item $\Phi$ satisfies the \emph{strong open set condition}: there exists a bounded open set $U$ with $U \cap F \neq \emptyset$, $S^{(i)}( U) \subset U$ for all $i \in \I$ and the images $S^{(i)} (U)$ are pairwise disjoint and
\item  the group generated by the set of normalised matrices $\tilde{\A}=\{\frac{1}{\sqrt{|\det A^{(i)}|}} A^{(i)}: i \in \I\}$ is strongly irreducible and is not contained in a compact subgroup of  $\mathcal{GL}_2(\R)$ (where by strongly irreducible we mean that the matrices do not preserve a finite union of lines through the origin in $\R^2$).
\end{enumerate}
Then $\hd F=\bd F= \dim \A$. \label{bhr}
\end{thm}

In this paper we will be concerned with the regularity of the dependence of the affinity dimension on the underlying maps. We return briefly to the definition of the similarity dimension for comparison. It is clear from (\ref{hm}) that the pressure $P_{\mathbf{r}}(s)$ depends continuously on the maps in the IFS, and indeed is analytic in $s$ and in each contraction ratio $r_i$ on $(0,1)$.  Via a suitable version of the implicit function theorem one may deduce that the similarity dimension of a self-similar set depends analytically on the contraction ratios. In this paper we would like to consider the analogous properties of the affinity dimension. 

According to the survey of Shmerkin \cite{shmerkin}, the question of the continuity of $\A \to \dim \A$  was a folklore open problem within the fractal geometry community since around 2000. The question was first raised explicitly in the papers of Falconer and Sloan \cite{fs} and Kaenmaki and Shmerkin \cite{ks}, where $P_{\A}(s)$ was shown to depend continuously on $\A$ in some special cases. However it was not until 2014 that $\A \mapsto \dim \A$ was shown to be continuous in general by Feng and Shmerkin \cite{fsh}. An alternative proof was subsequently given by the second named author \cite{morris}. It is a natural question to ask whether we can say anything stronger about the regularity of the map $\A \to \dim \A$. In this paper we will explore this question in the two dimensional setting. 

As an analogue of the analyticity of $P_{\mathbf{r}}(s)$ in the contraction ratios, in this paper we show that  $P_{\A}(s)$ is locally analytic in the matrix coefficients whenever the matrices strictly preserve a common cone and do not all preserve the same line going through the origin. For the purposes of this article we shall say that a closed convex subset $\mathcal{C} \subset \R^2 \setminus \{0\}$ is a \emph{cone} if for all $x \in \mathcal{C}$ and $\lambda>0$ we have $\lambda x \in \mathcal{C}$, and if $\textrm{int}  (\mathcal{C}) \neq \emptyset$. We say that the set of matrices $\mathcal{A}$ \emph{strictly preserves} the cone $\mathcal{C}$ if $A (\mathcal{C}\setminus\{0\}) \subset \textrm{int}(\mathcal{C})$ for all $A \in \mathcal{A}$. It is easy to see (in the two-dimensional context) that this is equivalent to the existence of a common basis with respect to which all of the matrices in $\A$ have positive entries. We also show that if $\A$ strictly preserves a common cone, $P_{\A}(s)$ is piecewise analytic in $s$, a property which was previously investigated in the context of triangular matrices by Fraser \cite{fraser}. Consequently we are able to show that the affinity dimension is locally analytic in the matrix coefficients.

Without loss of generality we can assume that $\I=\{1, \ldots, |\I|\}$ where $|\I|$ denotes the size of the alphabet $\I$. Given $\t=(t_1, \ldots, t_{4|\I|}) \in  \C^{4|\I|}$ and $k \in \I$ we let $A^{(k)}_{\t}$ denote the matrix
\begin{eqnarray}A^{(k)}_{\t}= \begin{pmatrix} t_{4k-3} & t_{4k-2} \\ t_{4k-1} &t_{4k} \end{pmatrix} \label{perturb} \end{eqnarray}
and $\A_{\t}=\{A^{(k)}_{\t}: k \in \I\}$. 

  Let $\t \in (-1,1)^{4|\I|}$ such that $\A_{\t}$ is a set of contracting invertible $2 \times 2$ matrices. Then we can define the sub-additive pressure associated to the parameters $s$ and $\t$ by
$$P(s,\t):=P_{\A_{\t}}(s)= \lim_{n\to\infty}\left(\sum_{A \in \A_{\t}^n}\phi^s(A )\right)^{\frac{1}{n}}.$$

We say that $\A_{\t}$ is \emph{irreducible} if there does not exist a one-dimensional subspace of $\R^2$ which is preserved by all of the matrices $A \in \A_{\t}$. The following is our main result.

\begin{thm}[Main theorem] \label{main} Let $\t_0 \in (-1,1)^{4|\I|}$ such that $\A_{\t_0}$ is an irreducible set of invertible matrices that preserve a common cone and are contracting with respect to some norm on $\mathbb{R}^2$, and suppose that $s_0= \dim \A_{\t_0} \in (0,1)\cup(1,2)$. Then there exists an open neighbourhood $U \subset (-1,1)^{4|\I|}$ of $\t_0$ and $U' \subset \R$ of $s_0$ such that $(s,\t) \mapsto P(s, \t)$ is analytic on $U' \times U$. Moreover, $\t \mapsto \dim \A_{\t}$ is analytic on $U$. \end{thm}

In particular, when $0<s < 1$, we have $\phi^s(A)= \norm{A}^s$. Therefore a special case of Theorem \ref{main} is the analyticity of the matrix norm pressure 
$$P(s,\t)= \lim_{n\to\infty}\left(\sum_{A \in \A_{\t}^n}\norm{A}^s \right)^{\frac{1}{n}}$$
in $(s,\t)$.

Since we are in the planar setting and the assumptions on our set of matrices imply that the assumptions of theorem \ref{bhr} are satisfied, our main result yields some corollaries concerning the analyticity of the Hausdorff (and box) dimensions. Fix a set of two-dimensional translation vectors $\{b^{(i)}: i \in \I\}$ and define $\Phi_{\t}:=\{S_{\t}^{(i)} (\cdot)=A^{(i)}_{\t}(\cdot)+b^{(i)}: i \in \I\}$ to be the iterated function system associated to the set of matrices $\A_{\t}$ and the translation vectors $\{b^{(i)}: i \in \I\}$. Let $F_{\t}=\bigcup_{i \in \I}  S^{(i)}_{\t} (F_{\t})$ denote the attractor of $\Phi_{\t}$. We say that $\Phi_{\t}$ satisfies the \emph{strong separation condition} if the pieces $S_{\t}^{(i)} (F_{\t})$ are pairwise disjoint.

\begin{cor} \label{bhr cor}
Let $\t_0$ satisfy the assumptions of theorem \ref{main} and suppose that the translations are chosen in a way that $\Phi_{\t_0}$ satisfies the strong separation condition. Then there exists an open neighbourhood $U \subset (-1,1)^{4|\I|}$ of $\t_0$ such that $\t \mapsto \hd F_{\t}=\bd F_{\t}$ is analytic on $U$. 
\end{cor}

Corollary \ref{bhr cor} follows immediately from theorems \ref{bhr} and \ref{main} and its proof can be found at the end of section \ref{proof}. $\Phi_{\t_0}$ is assumed to satisfy the strong separation condition rather than the strong open set condition to ensure that $\Phi_{\t}$ also has sufficient separation for $\t$ close to $\t_0$, in order for the hypothesis of theorem \ref{bhr} to be satisfied for $\t$ close to $\t_0$. 

Alternatively we can also allow the translations to vary with $\t$. For each $i \in \I$ and $\t$ associate a translation vector $b^{(i)}_\t \in \R^2$ and define $\Phi_{\t}:=\{S_{\t}^{(i)}(\cdot)=A_{\t}(\cdot)+b^{(i)}_{\t}: i \in \I\}$ with attractor $F_{\t}$.

\begin{cor} \label{cor2}
Let $\t_0$ satisfy the assumptions of theorem \ref{main} and suppose that $\Phi_{\t}$ satisfies the strong open set condition for any $\t$ in an open neighbourhood of $\t_0$. Then there exists an open neighbourhood $U \subset (-1,1)^{4|\I|}$ of $\t_0$ such that $\t \mapsto \hd F_{\t}=\bd F_{\t}$ is analytic on $U$.  \end{cor}

\section{Preliminaries}\label{prel}

In this section we restrict our attention to matrices of dimension $2$. Suppose that $\A=\{A^{(i)}: i \in \I\}$ strictly preserves a cone $\mathcal{C}$. Then there exists a matrix $B$ and a set of positive matrices $\mathcal{M}=\{M^{(i)}: i \in \I\}$ such that for each $i \in \I$, $M^{(i)}=B A^{(i)} B^{-1}$. Moreover, there exists a constant $C>0$ that depends only on $s$ and $B$ such that for any $n \in \N$ and $i_j \in \I$,
$$\frac{1}{C}\phi^s(A^{(i_1)} \cdots A^{(i_n)})\leq\phi^s(M^{(i_1)} \cdots M^{(i_n)})= \phi^s(B^{-1}A^{(i_1)} \cdots A^{(i_n)}B) \leq C\phi^s(A^{(i_1)} \cdots A^{(i_n)}).$$
Therefore $P_{\A}(s)=P_{\M}(s)$ for all $s$. Now, using the notation of the previous section, fix $\t_0 \in (-1,1)^{4|\I|}$ such that $\A_{\t_0}=\A$ and for $\t$ in a neighbourhood of $\t_0$ let $\M_{\t}$ denote the set of matrices in $\A_{\t}$ which have been conjugated by $B$ as above, so that $\M_{\t_0}=\M$. Write $\u_0$ to be the entries of the matrices in $\M_{\t_0}$. It is easy to see that the entries $\textbf{u}$ of the matrices in $\M_{\t}$ are linear combinations of the entries in $\t$. Therefore if $P_{\M_{\t}}(s)$ is analytic in $\mathbf{u}$ in some neighbourhood of $\mathbf{u}_0$, it follows that $P_{\A_{\t}}(s)=P_{\M_{\t}}(s)$ is analytic in $\t$ in some neighbourhood of $\t_0$. Therefore it is sufficient to prove theorem \ref{main} under the assumption that $\A_{\t_0}$ is a set of positive matrices.

Let $\A=\{A^{(i)}: i \in \I\}$ be a set of invertible positive matrices. For $k \in \mathbb{N}$ let $\mathcal{I}^k$ denote words of length $k$ over the alphabet $\mathcal{I}$ and let
\[
\mathcal{I}^* = \bigcup_{k \in \mathbb{N}} \mathcal{I}^k
\]
be the set of all finite words over the alphabet $\mathcal{I}$.  

 Also, let
$$\A^n=\{A_{i_1} \cdots A_{i_n}: i_j \in \I\}$$
and $\A^{\ast}=\bigcup_{n \in \N} \A^n$. For $\i = (i_1, i_2, \dots, i_k) \in \mathcal{I}^k$, write
\[
A^{(\i)} \ = \ A^{({i_1})}  \cdots  A^{({i_k})}.
\]

\subsection{Hardy-Hilbert space} Let $D$ be a disc of radius $\rho$ centred at $c \in \C$. The Hardy-Hilbert space $H^2(D)$ consists of all functions $f$ which are analytic on $D$ and such that $\sup_{r< \rho} \int_0^1 |f(c+re^{2\pi i t})|^2 \textup{d}t< \infty$. The inner product on $H^2(D)$ is defined by 
$$\langle f, g \rangle_{H^2}= \int_0^1 f(c+re^{2\pi i t})\overline{g(c+e^{2\pi i t})} \textup{d}t$$
which is well-defined since any element of $H^2(D)$ extends as an $L^2$ function of the boundary $\partial D$. The norm of $f \in H^2(D)$ is then given as $\norm{f}_{H^2}=\langle f,f \rangle_{H^2}^{\frac{1}{2}}$. 

An alternative characterisation of $H^2(D)$ is given as the space of all functions $f$ which are analytic on $D$ which can be expressed in the form
$$f(z)= \sum_{n=0}^{\infty} \alpha_k(f) \frac{(z-c)^k}{\rho^k}$$
for some square-summable sequence of complex numbers $\{\alpha_k(f)\}_{k=0}^{\infty}$. We will primarily utilise this second characterisation of $H^2(D)$.
This second characterisation permits us to write the norm of $f \in H^2(D)$ alternatively as
$$\norm{f}_{H^2}= \left(\sum_{k=0}^{\infty} |\alpha_k(f)|^2\right)^{\frac{1}{2}}.$$

If $b:D \to \C$ is bounded and analytic on $D$ and $f \in H^2(D)$ then $bf \in H^2(D)$ and
\begin{eqnarray}
\norm{bf}_{H^2} \leq \norm{b}_{\infty} \norm{f}_{H^2}, \label{bf}
\end{eqnarray}
see \cite[\S 1.2]{shapiro}. In particular if $f$ is bounded and analytic on $D$ then $f \in H^2(D)$ and $\norm{f}_{H^2} \leq \norm{f}_{\infty}$. 

Throughout the rest of this paper we fix $D$ to be the disc of radius $\frac{1}{2}$ centred at $\frac{1}{2}$.

\subsection{Perturbation theory}

Let $F:\C^n \to \C$ be a function. We recall that $F$ is called (complex) analytic in a neighbourhood $U \subset \C^n$ if for each $(a_1, \ldots, a_n) \in U$ one can write
$$F(z_1, \ldots, z_n)= \sum_{k_1, \ldots, k_n \in \N} c_{k_1, \ldots, k_n} (z_1-a_1)^{k_1} \cdots (z_n-a_n)^{k_n}$$
where $c_{k_1, \ldots, k_n} \in \C$ and the series is convergent to $F(z_1, \ldots, z_n)$ for all
 $(z_1, \ldots, z_n) $ in a neighbourhood of $(a_1, \ldots, a_n)$. By Hartogs's theorem a function $F:\C^n \to \C$ is (complex) analytic if and only if it is (complex) analytic in each variable separately. 

Let $U \subset \C$ be an open neighbourhood, $B$ be a Banach space of functions equipped with a norm $\norm{\cdot}$ and $L_{t}:B \to B$ be operators for each $t \in U$. We say that $\{L_t\}_{t \in U}$ is an analytic family of operators if for each $a \in U$ there exists a constant $0<r<1$ and there exist operators $L_k: B \to B$ with $\norm{L_k} =O(r^{-k})$ such that
$$L_{t}= \sum_{k \in \N}  (t-a)^{k} L_{k}$$
for all $t \in B(a, r)$, where convergence of the series is understood in the sense of the operator norm topology.

 The following perturbation theorem is a special case of the more general analytic perturbation theorem presented in \cite[Theorem 3.8]{hennion}.

\begin{prop}[Analytic perturbation theorem] \label{apt} Let $U \subset \C$ be an open neighbourhood of $t_0 \in \C$ and $\{L_{t}\}_{t \in U}$ be an analytic family of bounded linear operators on a Banach space. Suppose $L_{t_0}$ has a maximal eigenvalue $\lambda_1(L_{t_0})$ which is separated from the rest of the spectrum of $L_{t_0}$ and $\lambda_1(L_{t_0})$ is (algebraically \emph{and} geometrically) simple. Then there exists an open neighbourhood $U' \subset U$ of $t_0$ such that for all $t \in U'$, $L_t$ also has a maximal simple eigenvalue $\lambda_1(L_t)$ which is separated from the rest of the spectrum of $L_t$. Moreover $\lambda_1(L_t)$ is an analytic function of  $t \in U'$. \end{prop}

We will also require the following perturbation theorem.

\begin{prop}
Suppose $L:B \to B$ is a bounded linear operator on a Banach space that has a simple eigenvalue of maximum modulus which is separated from the rest of the spectrum of $L$. Then there exists $\epsilon>0$ such that for all bounded linear operators $T:B \to B$ with the property that
$$\norm{L-T}< \epsilon,$$
the operator $T$ also has a simple eigenvalue of maximum modulus which is separated from the rest of the spectrum of $T$.
\label{kato}
\end{prop}

\begin{proof}
Follows from \cite[Theorems IV.2.14 and IV.3.16]{kato}.
\end{proof}

\subsection{Transfer operator} \label{op} 

If $A$ is an invertible $2 \times 2$ matrix, we have a simpler characterisation of $\phi^s$ given by
\begin{equation}
\phi^s(A) =\begin{cases} \norm{A}^{s} & s \in [0,1)  \\
\norm{A}^{2-s} \, |\det(A)|^{s-1} & s \in [1, 2]\\ \end{cases} \label{simpler}
\end{equation}
where we have used the identities $\alpha_1(A)=\|A\|$ and  $|\det A|=\alpha_1(A)\alpha_2(A)$.

Given a (real or complex valued) matrix $A= \begin{pmatrix} a&b \\c&d \end{pmatrix}$ we define the map $w_A:D \to \C$ by
$$w_A(z)= (a+c-b-d)z+b+d.$$
If $A$ is a matrix such that $w_A(z) \neq 0$ on $D$ then we also define the function $\phi_A: D \to \C$ by
$$\phi_A(z)= \frac{(a-b)z+b}{(a+c-b-d)z+b+d}.$$
If $A$ is a real positive matrix, this can be understood as the action of the matrix on the first co-ordinate of representative vectors in $\R\P_+^1$, where $\R\P_+^1$ denotes the space of positive directions in $\R^2$. In particular let $\Delta=\{(x, 1-x): x \in (0,1)\}$ which are representative vectors for $\R\P_+^1$. Then
$$
\begin{pmatrix} \phi_A(x) \\1-\phi_A(x) \end{pmatrix}= \frac{1}{w_A(x)} A \begin{pmatrix} x \\1-x \end{pmatrix} \in \Delta.
$$
Moreover for any $\x=(x,1-x) \in \Delta$, $w_A(x)= \langle A\x, \u \rangle$ where $\u:=(1,1)$. 

Throughout the rest of this section we make the following assumption on the parameter $\t \in \C^{4|\I|}$. Let $\C_{+}$ denote the right half plane. 
\begin{ass}
We assume that $\t \in \C^{4|\I|}$ satisfies: 
\begin{enumerate}[(i)]
\item $\Re(t_{4k-3}t_{4k}-t_{4k-2}t_{4k-1}) \neq 0$ for all $k \in \I$,
\item $ \overline{\phi_{A_{\t}^{(k)}}(D)} \subset D$ for all $k \in \I$, 
\item $\overline{w_{A_{\t}^{(k)}}(D)} \subset \C_{+}$ for all $k \in \I$ and
\item for some distinct $i, j \in \I$, $\phi_{A_{\t}^{(i)}} \neq \phi_{A_{\t}^{(j)}}$.
\end{enumerate}
We denote the set of $\t$ that satisfies (i)-(iv) by $\Omega$. \label{ass}
\end{ass}

\begin{rem} Suppose $\A_{\t}$ is an irreducible set of positive invertible matrices. Then (i) is satisfied because the determinant of each matrix is non-zero and (iv) is satisfied because $\A_{\t}$ is irreducible. Denoting $A= \begin{pmatrix} a&b\\c&d \end{pmatrix} \in \A_{\t}$, then (ii) is satisfied because $\phi_{A}$ maps $D$ to a disk centred in the real axis whose boundary passes through the points $0< \frac{a}{a+c}, \frac{b}{b+d}<1$ and (iii) is satisfied because $w_A$ maps $D$ to a disk centred in the real axis whose boundary passes through the points $a+c$ and $b+d$. Therefore $\t \in \Omega$. In fact if it is assumed that $\t \in \Omega \cap (0,1)^{4|\I|}$ then $\A_{\t}$ is necessarily an irreducible set of positive invertible matrices. \end{rem}

For $\t \in \Omega$ and $A \in \A_{\t}$ define the composition operator $\c_A: H^2(D) \to H^2(D)$ by $\c_A f=f \circ \phi_A$. Note that since $\phi_{A}$ is an analytic self-map of $D$, by Littlewood's theorem \cite[page 11]{shapiro} $\c_A$ preserves $H^2(D)$. Observe that for any $A_1, \ldots , A_n \in \A_{\t}$, 
\begin{eqnarray}
\phi_{A_1 \cdots A_n}= \phi_{A_1} \circ \cdots \circ \phi_{A_n}.
\label{phi eqn}
\end{eqnarray}
Also notice that for any $A_1, \ldots , A_n \in \A_{\t}$,
\begin{eqnarray}
w_{A_1 \cdots A_n}=(w_{A_1}\circ\phi_{A_2 \cdots A_n})(w_{A_2}\circ \phi_{A_3 \cdots A_n}) \cdots (w_{A_{n-1}} \circ \phi_{A_n}) w_{A_n}. \label{w eqn}
\end{eqnarray}

Let $s \in \{z \in \C: |z| \leq 2\}$, $\t \in \Omega$ and $k \in \I$. Define
\begin{equation}
\psi_{A^{(k)}_{\t},s}(z)=\begin{cases} w_{A_{\t}^{(k)}}(z)^s & \text{if}\ 0\leq |s| \leq 1 \\
w_{A^{(k)}_{\t}}(z)^{2-s}(t_{4k-3}t_{4k}-t_{4k-2}t_{4k-1})^{s-1} & \text{if $1<|s|\leq 2$ and $\Re(t_{4k-3}t_{4k}-t_{4k-2}t_{4k-1})>0$} \\
w_{A^{(k)}_{\t}}(z)^{2-s}(t_{4k-2}t_{4k-1}-t_{4k-3}t_{4k})^{s-1} &\text{if $1<|s|\leq 2$ \text{and} $\Re(t_{4k-2}t_{4k-1}-t_{4k-3}t_{4k})>0$} \end{cases}
\end{equation}
where $f(z)^s=\exp(s \log f(z))$ where $\log$ is understood as the unique analytic function from $\C_{+}$ to $\C$ such that $\exp \log z=z$. For each $\i \in \I^{\ast}$, $\psi_{A_{\t}^{(\i)},s}$ can be defined analogously. Then for $A \in \A_{\t}^{\ast}$ define the multiplication operator $\M_{A,s}: H^2(D) \to H^2(D)$ by $\M_{A,s}f=\psi_{A,s} \cdot f$. $\M_{A,s}$ preserves $H^2(D)$ since $\psi_{A,s}$ is bounded and analytic on $D$ \cite[page 11]{shapiro}.

Finally, for $s \in\{z \in \C: |z| \leq 2\}$ and $\t \in \Omega$  we define the weighted composition operator $\l_{s,\t}: H^2(D) \to H^2(D)$ by
$$\l_{s,\t} f(z)=\sum_{A \in \A_{\t}} \M_{A,s}\c_{A} f(z)= \sum_{A \in \A_{\t}} \psi_{A,s}(z) f( \phi_{A}(z)).$$

Notice that by (\ref{phi eqn}), (\ref{w eqn}) and the fact that the determinant is a multiplicative functional, the iterates of $\l_{s,\t}$ are given by
$$\l^n_{s,\t} f(z)= \sum_{A \in \A_{\t}^n} \psi_{A,s}(z) f( \phi_{A}(z)).$$

The following simple observation will allow us to relate the spectrum of $\l_{s,\t}$ with $\dim \A_{\t}$ whenever $\t \in (0,1)^{4|\I|}\cap \Omega$ and $s \in [0,2]$.

\begin{lma} \label{equivalence}
Fix $\t \in (0,1)^{4|\I|}\cap \Omega$. There exists a constant $c>0$ that depends only on $\t$ and $s$ such that for all $A \in \A_{\t}^{\ast}$  and $x \in (0,1)$,
\begin{eqnarray}
c^{-1}\phi^s(A) \leq \psi_{A,s}(x) \leq c \phi^s(A).\label{equiv}
\end{eqnarray}
\end{lma}

 \begin{proof}
 Observe that by definition of $\psi_{A,s}$ and the characterisation of $\phi^s$ given in (\ref{simpler}), it is sufficient to show that there exists a constant $c>0$ that depends only on the set $\A_{\t}$ such that for all $x \in (0,1)$ and $A \in \A_{\t}^{\ast}$
\begin{eqnarray}
c^{-1}\norm{A} \leq \langle A\x,\u \rangle \leq c\norm{A}\label{dot eqn} \end{eqnarray}
\label{dot norm}
where $\x=(x,1-x)$ and $\u=(1,1)$. Fix $A \in \A_{\t}^{\ast}$. To verify the right hand side, notice that by the Cauchy-Schwarz inequality,
$$|\langle \mathbf{x}, \mathbf{u}\rangle | \leq \sqrt{2} \norm{\mathbf{x}}$$
and therefore since $x <1$,
$$|\langle A\mathbf{x}, \mathbf{u} \rangle| \leq \sqrt{2}\norm{A\mathbf{x}} \leq \sqrt{2}\norm{A}.$$

To verify the left hand side we begin by claiming there exist uniform constants $\epsilon , \delta>0$ such that
\begin{eqnarray}
|\langle A\mathbf{x}, \mathbf{u}\rangle | &\geq& \epsilon \norm{A \mathbf{x}} \label{epsilon} \\
\norm{A \mathbf{x}} &\geq & \delta \norm{A} \norm{\mathbf{x}} \label{delta}
\end{eqnarray}
which are independent of the choice of $\x$ and $A$. Observe that it is enough to show that (\ref{epsilon}) and (\ref{delta}) hold uniformly for any $\mathbf{x}$ and $A$ with $\norm{\mathbf{x}}=\norm{A}=1$. By compactness of $[0,1]$ and continuity of $x \mapsto |\langle A\x, \mathbf{u} \rangle|$ and $x \mapsto \norm{A \x}$ it is sufficient to show that $|\langle A\mathbf{x}, \mathbf{u}\rangle| \neq 0$ and $\norm{A \mathbf{x}} \neq 0$ which both clearly hold. Therefore there exist uniform constants $\epsilon, \delta>0$ such that (\ref{epsilon}) and (\ref{delta}) hold. Therefore for all $x \in (0,1)$ and $\mathbf{x}=(x, 1-x)$,
\begin{eqnarray*}
|\langle A\mathbf{x}, \mathbf{u}\rangle| &\geq& \epsilon \norm{A \mathbf{x}} \\
&\geq& \epsilon \delta \norm{A} \norm{\mathbf{x}} \\
&\geq& \frac{\epsilon \delta}{\sqrt{2}} \norm{A}|\langle\mathbf{x}, \mathbf{u}\rangle|= \frac{\epsilon \delta}{\sqrt{2}} \norm{A}.
\end{eqnarray*}
\end{proof}

Since the work of Ruelle \cite{ruelle}, it has been well understood that analytic weighted composition operators acting on spaces of analytic functions have strong spectral properties. By invoking for instance \cite[Proposition 2.10]{oj} we can deduce that $\l_{s,\t}$ is a compact operator whenever $\t \in\Omega$  and $s \in \{z: |z| \leq 2\}$. We recall the following version of the Krein-Rutman theorem. 

\begin{prop}
Let $X$ be a Banach space and $K \subset X$ be closed convex set such that
\begin{enumerate}[(i)]
\item $\lambda K \subset K$ for all $\lambda \geq 0$,
\item $K \cap (-K)=\{0\}$ and
\item $\interior{K} \neq \emptyset$.
\end{enumerate}
Assume that $L:X \to X$ is a compact linear operator such that $LK \subseteq K$. Suppose that for all $f \in K \setminus \{0\}$ there exists $n \in \N$ such that $L^nf \in \interior{K}$. Then its spectral radius $\rho(L)>0$ and $\rho(L)$ is a simple eigenvalue with an eigenfunction $f \in K$. Moreover, $L$ does not have any other eigenvalues of modulus $\rho(L)$.
\label{krt}
\end{prop}

\begin{proof}
The existence of $\lambda>0$ and $f \in K$ such that $Lf=\lambda f$ follows from \cite[Theorem 2.5]{kras}. The fact that it is a simple eigenvalue follows from \cite[Theorem 2.10]{kras}. The fact that $\lambda$ is a unique eigenvalue of maximum modulus follows from \cite[Theorem 2.13]{kras}.
\end{proof}

\begin{lma} \label{ev}
Let $\t \in \Omega \cap (0,1)^{|\I|}$ so that $\A_{\t}$ is an irreducible set of invertible positive matrices and let $s \in [0,2]$. Then:
\begin{enumerate}[(a)]
\item  
There is a unique eigenvalue of maximum modulus for $\l_{s,\t}$, which we denote by $\lambda_1(s,\t)$. It is a simple eigenvalue.
\item $\lambda_1(s,\t)=P(s,\t)$.
\end{enumerate}
\end{lma}

\begin{proof}
We begin by proving (a). Fix $\t$ and $s$. For all $n \in \N$ define 
$$\Gamma^n= \overline{ \bigcup_{A \in \A_{\t}^n} \phi_{A}(D)}$$
and 
$\Gamma= \bigcap_{n=1}^{\infty} \Gamma^n \subset (0,1)$. Since $\Gamma^n$ is a nested sequence of closed subsets of $D$, $\Gamma$ is a compact subset of $D$. Moreover $\A_{\t}$ is irreducible since $\t$ satisfies condition (iv) of assumption \ref{ass}, and this guarantees that $\Gamma$ is an infinite set of points. Define
$$X=\{f \in H^2(D): \textnormal{$f(z) \in \R$ for all $z \in \Gamma$}\}$$
and $K \subset X$ as
$$K=\{f \in H^2(D): \textnormal{$f(z) \geq 0$ for all $z \in \Gamma$}\}.$$
It is easy to see that $(X, \norm{\cdot}_{H^2})$ is a real Banach space, that $\l_{s,\t}X \subseteq X$, $\l_{s,\t}K \subseteq K$ and that $K$ is a closed convex set that satisfies (i)  of proposition \ref{krt}. $K$ satisfies (ii) since any holomorphic function that is zero on a compact infinite set is the zero function. Since $\l_{s,\t}$ is a compact operator on $H^2(D)$ it is easy to see that its action on $X$ is also compact.

Let $f \in H^2(D)$. By the Cauchy-Schwarz inequality
\begin{eqnarray*}
\sup_{z \in \Gamma} |f(z)| &\leq& \left(\sum_{n=0}^{\infty} \alpha_n(f)^2 \right)^{\frac{1}{2}} \left(\sum_{n=0}^{\infty} 2^{2n}(z-\frac{1}{2})^{2n}\right)^{\frac{1}{2}} \\
&\leq& \norm{f}_{H^2} \left(\sum_{n=0}^{\infty}\epsilon^{2n}\right)^{\frac{1}{2}}
\end{eqnarray*}
where $\epsilon:= \sup_{z \in \Gamma}|2(z-\frac{1}{2})|<1$ since $\Gamma$ is a compact subset of $(0,1)$. Therefore writing $C=\left(\sum_{n=0}^{\infty}\epsilon^{2n}\right)^{\frac{1}{2}}$ we have that for any $f \in H^2(D)$,
\begin{eqnarray}
\sup_{z \in \Gamma}|f(z)| \leq C\norm{f}_{H^2}.
\label{int eqn}
\end{eqnarray}
It follows easily that the open $1/C$-ball with centre $\one$ is a subset of $K$ and in particular $\one \in \interior{K}$ so that (iii) is satisfied. 

Next we check that for each $f \in K \setminus \{0\}$ there exists some $n \in \N$ such that $\l_{s,\t}^nf \in \interior{K}$. By (\ref{int eqn}) it is sufficient to show that $\l_{s,\t}^nf>0$ on $\Gamma$. If $f>0$ on $\Gamma$ then by positivity of $\psi_{A_{\t},s}$ it follows that $\l_{s,\t} f>0$ on $\Gamma$. If $f$ is not positive on $\Gamma$ it may have only finitely many zeroes within $\Gamma$. We claim that there exists $n$ sufficiently large that $\l^n_{s,\t}f$ has at most one zero within $\Gamma$. To see this, choose $M$ sufficiently large that for all  $A \in \A_{\t}^M$, $\phi_{A}(\Gamma)$ has sufficiently small diameter so that it can contain at most one zero of $f$ (which is possible since $f$ has only finitely many zeroes). In particular, for each $A \in \A_{\t}^M$, $\psi_{A,s}\cdot f \circ \phi_A$ has at most one zero within $\Gamma$. Therefore, 
$$\l^M_{s,\t}f(x)= \sum_{A \in \A^M_{\t}} \psi_{A_{\t},s}(x)f(\phi_{A_{\t}}(x))$$
has at most one zero within $\Gamma$, which we denote by $x_0 \in \Gamma$. For $A \in \A_{\t}^{\ast}$, let $x_A \in \Gamma$ denote the unique fixed point of $\phi_A$. Note that $x_A=x_{A^n}$ for all $A \in \A_{\t}^{\ast}$ and $n \geq1$ and that by irreducibility of $\A_{\t}$, there must exist $A, B \in \A_{\t}$ such that $x_A \neq x_B$. Now choose any $A \in \A_{\t}$ such that $x_A \neq x_0$ and choose $N$ sufficiently large that $x_0 \notin \phi_{A^n}(\Gamma)$ for all $n \geq N$. Then it follows that no $x \in \Gamma$ can be a zero of $f\circ \phi_{A^n}$ and therefore for all $x \in \Gamma$,
$$\l_{s,\t}^{M+N}f(x) \geq \psi_{A^{M+N},s}(x)f(\phi_{A^{M+N}}(x))>0$$
completing the proof of the claim.

Therefore, by applying proposition \ref{krt} we deduce that $\rho(\l_{s,\t}\bigr|_{X})>0$ and that there exists $h_{s,\t} \in K$ such that $\l_{s,\t}h_{s,\t}=\rho(\l_{s,\t}\bigr|_{X})h_{s,\t}$. Since $\rho(\l_{s,\t}\bigr|_{X})^n h_{s,\t}=\l_{s,\t}^nh_{s,\t} \in \interior{K}$ for some $n \in \N$, it follows that $h_{s,\t} \in \interior{K}$, in particular $h_{s,\t}$ is positive on $\Gamma$. 

Next we prove that $H^2(D)=X+iX$ which implies that the spectrum of $\l_{s,t}$ on $X$ is identical to its spectrum on $H^2(D)$, and the multiplicity of each of its eigenvalues is the same on both spaces. In particular this yields $\rho(\l_{s,\t}|_{H^2(D)})=\rho(\l_{s,\t}\bigr|_{X})$, so we may denote their common value by $\lambda_1(s,\t)$. To see that $H^2(D)=X+iX$ it is sufficient to show that $f \in X$ if and only if $f \in H^2(D)$ and $\alpha_n(f) \in \R$ for all $n \geq 0$. It is obvious that if $f \in H^2(D)$ has $\alpha_n(f) \in \R$ for all $n \geq 0$ then $f \in X$, so let us prove the converse direction.

Let $f \in X$. Then the function $g \colon D \to \mathbb{C}$ defined by $g(z):= \overline{f(\overline{z})}$ (where $\overline{z}$ denotes the complex conjugate of $z$) has power series given by $g(z)=\sum_{n=0}^\infty \overline{\alpha_n(f)}z^n$, and therefore belongs to $H^2(D)$. If $z \in \Gamma$ then since $f(z) \in \R$ and $z \in \Gamma \subset \mathbb{R}$ we have $g(z)=\overline{f(z)}=f(z)$. In particular, since any two holomorphic functions which coincide on a compact infinite subset of $D$ must necessarily coincide on $D$, we have $g(z)=f(z)$ for all $z \in D$. By comparing the power series of $f$ and $g$ this implies that $\overline{\alpha_n(f)}=\alpha_n(f)$ for all $n \geq 0$, and so $\alpha_n(f) \in \R$ for all $n \geq 0$ as required.

Next we prove part (b). Fixing $z \in \Gamma$ we have
$$\lambda_1(s,\t)= \lim_{n \to \infty} (\l_{s,\t}^nh_{s,\t}(z))^{\frac{1}{n}} = \lim_{n \to \infty} \left( \sum_{A \in \A_{\t}^n} \psi_{A,s}(z)h_{s,\t}(\phi_{A}(z))\right)^{\frac{1}{n}}.$$
Since $h_{s,\t} \in \interior{K}$ it follows that $h_{s,\t}$ is positive on $\Gamma$, so by compactness of $\Gamma$ there exists a constant $c'>0$ such that for all $ y \in \Gamma$, $\frac{1}{c'} \leq h_{s,\t}(y) \leq c'$. Since $\phi_A(z) \in \Gamma$ for all $A \in \A_{\t}^{\ast}$ and all $z \in \Gamma$, we can combine this bound with (\ref{equiv}) to imply that
$$\lambda_1(s,\t)= \lim_{n \to \infty}  \left( \sum_{A \in \A_{\t}^n} \phi^s(A)\right)^{\frac{1}{n}}= P(s,\t).$$
\end{proof}

\section{Proofs of results}

Fix $\t_0=(\tau_1, \ldots, \tau_{4|\I|}) \in (0,1)^{4|\I|} \cap \Omega$ with $s_0=\dim \A_{\t_0} \in (0,1)\cup (1,2)$. It is easy to see that for each $1 \leq i \leq 4|\I|$ there exist connected neighbourhoods $U_i \subset \C$ of $\tau_i$ with the property that $U_1 \times \cdots \times U_{4|\I|} \subset \Omega$ and a connected neighbourhood $V \subset \{z: 0< |z|<1\} \cup \{z: 1 < |z|<2\}$ of $s_0$.

The plan of the proof is as follows. For each $1 \leq i \leq 4|\I|$ we will show that $\{\l_{s,\t}\}$ is an analytic family in $t_i$ on $U_i$ whenever $s \in V$ is held constant and $t_j \in U_j$ is held constant for $j \neq i$. We will also show that $\{\l_{s,\t}\}$ is an analytic family in $s$ on $V$ whenever $t_i \in U_i$ are held constant for all $i$.  We will then invoke the perturbation theorems (propositions \ref{apt} and \ref{kato}) to deduce that $\lambda_1(s,\t)$ is an analytic function of $s$ in a neighbourhood $\tilde{V} \subset V$ of $s_0$ while $t_i \in U_i$ are held constant for all $i$ and that $\lambda_1(s,\t)$ is analytic in each $t_i$ on a neighbourhood $\tilde{U_i} \subset U_i$ of $\tau_i$ whenever $s \in V$ and $t_j \in U_j$ are held constant for all $i \neq j$. Hartogs's theorem will imply that $\lambda_1(s,\t)$ is jointly analytic in $(s,\t)$ on $\tilde{V} \times \tilde{U_1} \times \cdots \times \tilde{U}_{4|\I|}$, therefore $P(s,\t)$ is real analytic on $\tilde{V} \times \tilde{U_1} \times \cdots \times \tilde{U}_{4|\I|}\cap \R^{4|\I|+1}$.

 \subsection{Analyticity of the composition operator}
 
 Throughout this section we fix $t_j \in U_j$ for each $j$ and write $\t=(t_1, \ldots, t_{4|\I|})$. Given some $1 \leq i \leq 4|\I|$ and $t \in U_i$ we will let $\t_{i,t}$ denote the complex valued vector obtained by taking $\t$ and replacing $t_i$ by $t$.

 \begin{lma}
Fix some $l \in \I$ and $4l-3 \leq i \leq 4l$. There exist $C_0< \infty$, $0<r<1$ and analytic bounded functions $f_{k}:D \to \C$ with $\norm{f_{k}}_{\infty} \leq \frac{C_0^{k}r}{2}$ such that 
\begin{eqnarray*}\phi_{A^{(l)}_{\t_{i,t}}}-\frac{1}{2}=\sum_{k =0}^{\infty} (t-t_i)^{k} f_{k} \end{eqnarray*}
for all $t \in B(t_i, C_0^{-1}) $.  \label{phi holo}
\end{lma}

\begin{proof}
Notice that it is sufficient to prove the result for $l=1$ and $i \in \{1,2,3,4\}$. We begin by assuming $i=1$. Denote $g(z)=(t_1-t_2)z+t_2$, $G(z)=(t_1+t_3-t_2-t_4)z+t_2+t_4$. Let $C_0:= \sup_{z \in D}\left\{\left|\frac{z}{G(z)}\right|\right\}$.  Then for any $t \in B(t_1, C_0^{-1})$,
\begin{eqnarray*}
\phi_{A^{(1)}_{\t_{1,t}}}(z)&=& \frac{g(z)+(t-t_1)z}{G(z)+(t-t_1)z} \\
&=& \frac{g(z)}{G(z)} \cdot \frac{1}{1+\frac{(t-t_1)z}{G(z)}} + \frac{(t-t_1)z}{G(z)} \cdot \frac{1}{1+\frac{(t-t_1)z}{G(z)}}  \\
&=& \frac{g(z)}{G(z)} \sum_{n=0}^{\infty} \frac{(-1)^n}{G(z)^n} (t-t_1)^nz^n+\frac{(t-t_1)z}{G(z)} \sum_{n=0}^{\infty} \frac{(-1)^n}{G(z)^n} (t-t_1)^nz^n. 
\end{eqnarray*}
Therefore we can write $\phi_{A^{(1)}_{\t_{1,t}}}-\frac{1}{2}= \sum_{m =0}^{\infty} (t-t_1)^{m} f_{m}$ where $f_0(z)= \frac{g(z)}{G(z)}-\frac{1}{2}=\phi_{A^{(1)}_{\t}}-\frac{1}{2}$ and for $m \geq 1$, 
$$f_m(z)=\frac{(-z)^m}{G(z)^m}\left(\frac{g(z)}{G(z)}-1\right).$$
It is easy to see that the functions $f_m$ are analytic on $D$. In order to verify the uniform bound on $f_{m}$ it  is sufficient to check that $\norm{f_{0}}_{\infty} \leq \frac{r}{2}$ for some $0<r<1$ and that there exists some $C_0< \infty$ such that $\norm{f_{m}}_{\infty} \leq C_0^{m}$. For the first claim, since $\t \in \Omega$, $\overline{\phi_{A^{(1)}_{\t}}(D)} \subset D$, thus $\norm{f_0|_D}_{\infty}< \frac{1}{2}$. Therefore we can define $r:= 2\norm{f_{0}|_{D}}_{\infty} <1$. For the second claim, if we replace $C_0$ by $\frac{2C_0}{r}$ then it easily follows that $\norm{f_m}_{\infty} \leq C_0^m$.

For other values of $i$ the proof is almost identical, therefore we omit the details.
\end{proof}

We will require the following two technical lemmas.

\begin{lma}
Let $0<r<1$, $C_0< \infty$ and let $f_{k}: D \to \C$ be analytic bounded functions such that $\norm{f_{k}}_{\infty} \leq \frac{C_0^{k}r}{2}$ for all $k \in \N$. Then there exist analytic bounded functions $\varphi_{m,n}:D \to \C$ such that
$$\left(\sum_{k =0}^{\infty} x^{k}f_{k}\right)^n=\sum_{m =0}^{\infty} x^{m} \varphi_{m,n}$$
for all $x \in B(0, C_0^{-1})$. Moreoever $\varphi_{m,n}$ are independent of $x$ and
\begin{eqnarray}
\norm{\varphi_{m,n}}_{\infty} \leq \frac{C_0^{m}r^n}{2^n} \frac{(m+n-1)!}{m!(n-1)!}. \label{multi bound} \end{eqnarray} \label{multi lemma}
\end{lma}

\begin{proof}
We begin by fixing $n$, expanding $\left(\sum_{k=0}^{\infty} x^{k}f_{k}\right)^n$
and finding the coefficient of $x^{m}$. It is easy to see that this will coincide with the coefficient of $x^{m}$ in
$\left(\sum_{k=0}^m x^{k}f_{k}\right)^n$.
Denoting $x_{k}=x^{k}f_{k}$ and applying the multinomial theorem we see that 
\[\left(\sum_{k=0}^m x^{k}f_{k}\right)^n =
\sum_{ i_{0}+ \ldots +i_m=n} \frac{n!}{\prod_{k=0}^m i_{k}!} \prod_{k=0}^m x^{k i_{k}}f_{k}^{i_{k}}\\
=\sum_{ i_{0}+ \ldots +i_m =n} \frac{n!}{\prod_{k=0}^m i_{k}!} x^{  \sum_{k=0}^m ki_k} \prod_{k =0}^m f_{k}^{i_{k}}.
\]
Therefore, the coefficient of $x^{m}$ is given by
$$\varphi_{m,n}:= \sum_{\substack{ i_{0}+ \ldots +i_m=n \\ i_1+2i_2+ \ldots + m i_{m}=m}} \frac{n!}{\prod_{k =0}^m i_{k}!}  \prod_{k=0}^m f_{k}^{i_{k}}.$$
Now, to verify (\ref{multi bound}), notice that
\begin{eqnarray*}
\norm{\varphi_{m,n}}_{\infty} &\leq& \sum_{\substack{ i_{0}+ \ldots +i_m=n \\ i_1+2i_2+ \ldots + m i_{m}=m}} \frac{n!}{\prod_{k=0}^m i_{k}!}  \prod_{k=0}^m\norm{ f_{k}}_{\infty}^{i_{k}} \\
&\leq& \sum_{\substack{ i_{0}+ \ldots +i_m=n \\ i_1+2i_2+ \ldots + m i_{m}=m}} \frac{n!}{\prod_{k=0}^m i_{k}!}  \prod_{k=0}^m (\frac{C_0^{k}r}{2})^{i_{k}} \\
&=& \sum_{\substack{ i_{0}+ \ldots +i_m=n \\ i_1+2i_2+ \ldots + m i_{m}=m}} \frac{n!}{\prod_{k=0}^m i_{k}!} C_0^{\sum_{k=0}^m ki_{k}}(\frac{r}{2})^{\sum_{k=0}^m i_{k}} \\
&=& \sum_{\substack{ i_{0}+ \ldots +i_m=n \\ i_1+2i_2+ \ldots + m i_{m}=m}} \frac{n!}{\prod_{k=0}^m i_{k}!}  C_0^{m}\frac{r^{n}}{2^n}
\end{eqnarray*}
since for each term in the sum we have $\sum_{k=0}^m k i_{k}=m$ and $\sum_{k=0}^m i_{k}=n$. Therefore it remains to calculate
$$ \sum_{\substack{ i_{0}+ \ldots +i_m=n \\ i_1+2i_2+ \ldots + m i_{m}=m}} \frac{n!}{\prod_{k=0}^m i_{k}!}$$
which, by the multinomial theorem, is the coefficient of $x^{m}$ in the expansion of $\left(\sum_{k=0}^m x^{k}\right)^n$. This is given by $\frac{(m+n-1)!}{m!(n-1)!}$, see for example \cite[(7)]{moivre}. The result follows.
\end{proof}

\begin{lma} \label{geo bound}
Let $0<r<1$ and $k \in \N$. There exists $C_1>0$ (which depends on $k$ and $r$) for which
$$\sum_{n=0}^{\infty} r^n((n+1) \cdots (n+m))^k \leq C_1^m(m!)^k$$
for all $m \in \N$.
\end{lma}

\begin{proof}
We will prove the result by induction on $k$. Firstly, the claim is clearly true when $k=0$. Now, assuming it is true for $k-1$, we can write
\begin{eqnarray*}
\sum_{n=0}^{\infty} r^n((n+1) \cdots (n+m))^k&=& \sum_{n=0}^{\infty} \frac{\textup{d}^m}{\textup{d}r^m}\left(r^{n+m}((n+1)\cdots (n+m))^{k-1}\right) \\
&=& \frac{\textup{d}^m}{\textup{d}r^m} \left(r^m \sum_{n=0}^{\infty} r^n((n+1) \cdots (n+m))^{k-1}\right) \\
&=&\sum_{i=0}^m {m \choose i} \frac{\textup{d}^i}{\textup{d}r^i}(r^m) \frac{\textup{d}^{m-i}}{\textup{d}r^{m-i}}\left( \sum_{n=0}^{\infty} r^n((n+1) \cdots(n+m))^{k-1}\right) \\
&=&\sum_{i=0}^m {m \choose i} \frac{m!}{(m-i)!}r^{m-i} \frac{\textup{d}^{m-i}}{\textup{d}r^{m-i}}\left(\sum_{n=0}^{\infty} r^n((n+1) \cdots(n+m))^{k-1} \right).
\end{eqnarray*}
Put $f(z)=\sum_{n=0}^{\infty}z^n((n+1) \cdots (n+m))^{k-1}$ so that $f$ is defined and is analytic for all $|z|<1$. Put $r<r'<1$. By the Cauchy Integral Formula, for all $1 \leq j \leq m$ and $|w|<r$,
$$f^{(j)}(w)=\frac{j!}{2\pi i} \int_{|z|=r'} \frac{f(z)}{(z-w)^{j+1}} \textup{d}z$$
therefore
\begin{eqnarray*}|f^{(j)}(w)|& \leq&\frac{j!}{2\pi} \frac{\sup_{|z| \leq r'}|f(z)|}{(r'-r)^{j+1}} \cdot 2\pi r' \\
&\leq& \frac{j!\sup_{|z| \leq r'}|f(z)|}{(r'-r)^{j+1}} \\
&\leq& \frac{j! C_1^m (m!)^{k-1}}{(r'-r)^{j+1}}
\end{eqnarray*}
where the final line follows by the assumption on $k-1$. Therefore,
\begin{eqnarray*}
\sum_{n=0}^{\infty} r^n((n+1) \cdots (n+m))^k &\leq& \sum_{i=0}^m {m \choose i} \frac{m!}{(m-i)!}r^{m-i}  \frac{(m-i)! C_1^m (m!)^{k-1}}{(r'-r)^{m-i+1}} \\
&\leq& (m!)^k  \frac{C_1^m}{(r'-r)^{m+1}} \sum_{i=0}^m {m \choose i} .
\end{eqnarray*}
Since $\sum_{i=0}^m {m \choose i }=2^m$ the result follows.
\end{proof}

We now combine the last three lemmas prove that for each $k \in \I$ and $4k-3 \leq i \leq 4k$, $\{ \c_{A^{(k)}_{\t}}\}$ is analytic in $t_i$ on $U_i$, when $t_j \in U_j$ are held constant for $i \neq j$.

\begin{lma}
Fix some $l \in \I$ and $4l-3 \leq i \leq 4l$. There exists a constant $C_2< \infty$ and operators $\mathcal{P}_m: H^2(D) \to H^2(D)$ with $\norm{\mathcal{P}_{m}} \leq C_2^{m}$  such that $\c_{A^{(l)}_{\t_i}}= \sum_{m \in \N} (t-t_i)^{m} \mathcal{P}_{m}$ for all $t \in B(t_i, C_2^{-1})$. \label{ct}
\end{lma}

\begin{proof}
Let $f \in H^2(D)$. Fix some $l \in \I$ and $4l-3 \leq i \leq 4l$. Let $t$ belong to the neighbourhood of $t_i$ where lemma \ref{phi holo} is valid, and let $C_0$ and $r$ be as given by that lemma. By Lemmas \ref{multi lemma} and \ref{phi holo}, 
\begin{eqnarray*}
\c_{A^{(l)}_{\t_{i,t}}}(f)=f \circ \phi_{A^{(l)}_{\t_{i,t}}}&=& \sum_{n=0}^{\infty} \alpha_n(f)2^n (\phi_{A^{(l)}_{\t_{i,t}}}-\frac{1}{2})^n \\
&=& \sum_{n=0}^{\infty} \alpha_n(f)2^n \left(\sum_{k=0}^{\infty}  (t-t_i)^{k} f_{k}\right)^n \\
&=& \sum_{m =0}^{\infty} (t-t_i)^{m} \left(\sum_{n=0}^{\infty} \alpha_n(f)2^n \varphi_{m,n}\right).
\end{eqnarray*}
Define $\mathcal{P}_{m} f= \sum_{n=0}^{\infty} \alpha_n(f)2^n \varphi_{m,n}$. Since $\varphi_{m,n}$ are clearly analytic on $D$, in order to show that  $\mathcal{P}_m:H^2(D) \to H^2(D)$ are well defined operators it is sufficient to get an upper bound on $\norm{\mathcal{P}_m}$. Let $f \in H^2(D)$. Then since $\norm{\cdot}_{H^2} \leq \norm{\cdot}_{\infty}$ and by the Cauchy-Schwarz inequality,
\begin{eqnarray*}
\norm{\mathcal{P}_{m}f}_{H^2}&=&\bignorm{\sum_{n=0}^{\infty} \alpha_n(f) 2^n\varphi_{m,n}}_{H^2} \\
&\leq& \bignorm{\sum_{n=0}^{\infty} \alpha_n(f) 2^n\varphi_{m,n}}_{\infty} \\
&\leq& \left(\sum_{n=0}^{\infty}|\alpha_n(f)|^2\right)^{\frac{1}{2}} \left(\sum_{n=0}^{\infty} \norm{2^n\varphi_{m,n}}_{\infty}^2\right)^{\frac{1}{2}} \\
&=& \norm{f}_{H^2} \left(\sum_{n=0}^{\infty} \norm{2^n\varphi_{m,n}}_{\infty}^2\right)^{\frac{1}{2}} 
\end{eqnarray*}
and therefore $\norm{\mathcal{P}_{m}}_{H^2(D)} \leq \left(\sum_{n=0}^{\infty} 2^{2n}\norm{\varphi_{m,n}}_{\infty}^2\right)^{\frac{1}{2}} $. By Lemma \ref{multi lemma},
\begin{eqnarray*}
\left(\sum_{n=0}^{\infty} \norm{2^{2n}\varphi_{m,n}}_{\infty}^2 \right)^{\frac{1}{2}} &\leq& C_0^{m}\left(\sum_{n=0}^{\infty} 2^{2n}\frac{r^{2n}}{2^{2n}} \left( \frac{(m+n-1)!}{m!(n-1)!}\right)^2 \right)^{\frac{1}{2}} \\
&\leq& C_0^{m}  \frac{1}{m!} \left(\sum_{n=0}^{\infty} r^{2n}(n(n+1) \cdots (n+m-1))^2 \right)^{\frac{1}{2}} \\
&\leq& C_0^{m} C_1^{m}=(C_0C_1)^{m}
\end{eqnarray*}
where $C_1$ is fixed by Lemma  \ref{geo bound}.
\end{proof}

\subsection{Analyticity of weight function}

In the following lemma we establish the analyticity (in $t_i$ and $s$) of the weight function which appears in our transfer operator.

\begin{lma} \label{full holo}
Fix any $k \in \I$. 
\begin{enumerate}[(a)]
\item For each $4k-3 \leq i \leq 4k$ the map $\psi_{A^{(k)}_{\t},s}$ is analytic in $t_i$ on $U_i$ whenever $s \in V$ and $t_j \in U_j$ are held constant for $i \neq j$. In particular there exists a constant $C_3>0$ and functions $f_n \in H^2(D)$ with $\norm{f_n}_{\infty} \leq C_3^n$ such that
$$\psi_{A^{(k)}_{\t_{i,t}},s}= \sum_{n=0}^{\infty} f_n(t-t_i)^n$$ for all $t \in B(t_i, C_3^{-1})$. 
\item $\psi_{A^{(k)}_{\t},s}$ is analytic in s on $V$ whenever $t_i \in U_i$ are held constant for all $i$. In particular there exists a constant $C_4>0$ and functions $g_n \in H^2(D)$ with $\norm{g_n}_{\infty} \leq C_4^n$ such that 
$$\psi_{A^{(k)}_{\t},s}= \sum_{n=0}^{\infty} g_ns^n$$ for all $s \in V$. 
\end{enumerate}
\end{lma}

\begin{proof}
We begin with (a). It is sufficient to prove the result when $k=1$ and $i \in \{1,2,3,4\}$. We begin by assuming $i=1$ and $0<|s|<1$. Let $t \in U_1$. We denote $G(z)=(t_1+t_3-t_2-t_4)z+t_2+t_4$ so that
$$ \psi_{A^{(1)}_{\t_{1,t}},s}(z)=\exp(s \log w_{A^{(1)}_{\t_{1,t}}}(z))=\exp(s \log G(z))\exp\left(s\log\left(1+\frac{(t-t_1)z}{G(z)}\right)\right).$$

Therefore it is sufficient to show that $\log (1+ \frac{(t-t_1)z}{G(z)})$ can be written as a convergent power series in $(t-t_1)$. Indeed
$$\log\left(1+\frac{(t-t_1)z}{G(z)}\right)=\sum_{n=1}^{\infty} \frac{(-1)^n}{n} \left(\frac{z}{G(z)}\right)^n(t-t_1)^n
$$
which is valid for $|t-t_1|< \sup_{z \in D} \left|\frac{z}{G(z)}\right|< \infty$.

Next we assume that $1<|s|<2$. By definition of $U_1$ and $\Omega$, the real part of the determinant of $A_{\t_{1,t}}^{(1)}$ is the same sign for all $t \in U_1$. Therefore without loss of generality we can assume it is positive. Therefore
\begin{eqnarray}\psi_{A^{(1)}_{\t_{1,t}},s}(z)= G(z)^{2-s}(t_1t_3-t_2t_4)^{s-1}\exp((2-s)\log(1+\frac{(t-t_1)z}{G(z)}))\exp((s-1)\log(1+\frac{t_3(t-t_1)}{t_1t_3-t_2t_4}))\label{1} \end{eqnarray}
so it is sufficient to show that $\log(1+\frac{t_3(t-t_1)}{t_1t_3-t_2t_4})$ can be written as a convergent power series in $(t-t_1)$. Indeed
\begin{eqnarray}\log\left(1+\frac{t_3(t-t_1)}{t_1t_3-t_2t_4}\right)= \sum_{n=1}^{\infty} \frac{(-1)^n}{n} \left(\frac{t_3}{t_1t_3-t_2t_4}\right)^n(t-t_1)^n\label{2} \end{eqnarray}
which is valid for $|t-t_1|< \left|\frac{t_1t_3-t_2t_4}{t_3}\right|$. From (\ref{1}), (\ref{2}) and analyticity of $\exp (z)$ it is easy to deduce the existence of the analytic bounded functions $f_n$ that appear in the statement of the lemma and the exponential control on $\norm{f_n}_{\infty}$ is a consequence of the exponential control on the coefficients of $(t-t_1)^n$ which appear in (\ref{1}) and (\ref{2}). The proof of (a) for other values of $i$ is very similar and therefore we omit the details.
For part (b) the result follows directly from the fact that of $\exp(z)$ is an entire function of $z$.
\end{proof}

The following corollary summarises the consequences of Lemmas \ref{phi holo} - \ref{full holo} on our family of operators $\{\l_{s,\t}\}$. 

\begin{cor} \label{holo op}
Fix $\t_0=(\tau_1, \ldots, \tau_{4|\I|}) \in (0,1)^{4|\I|} \cap \Omega$ and $s_0=\dim \A_{\t_0} \in (0,1) \cup (1,2)$. For each $1 \leq i \leq 4|\I|$ there exist connected neighbourhoods $U_i \subset \C$ of $\tau_i$ with the property that $U_1 \times \cdots \times U_{4|\I|} \subset \Omega$  and a connected neighbourhood $V \subset \{z: |z|<1\} \cup \{z: 1<|z|<2\}$ of $s_0$ such that: 
\begin{enumerate}[(i)]
\item $\{\l_{s,\t}\}$ is an analytic family in $t_i$ on $U_i$ whenever $s \in V$ is held constant and $t_j \in U_j$ is held constant for $j \neq i$ and
\item $\{\l_{s,\t}\}$ is an analytic family in $s$ on $V$ whenever $t_i \in U_i$ are held constant for all $i$.
\end{enumerate}
\end{cor}

\begin{proof} We begin by proving (a). By lemma \ref{ct} and \ref{full holo}(a) we have
\begin{eqnarray}
\l_{A^{(k)}_{\t_{i,t}},s}&=& \left(\sum_{n=0}^{\infty} (t-t_i)^nf_n\right)\left(\sum_{m=0}^{\infty} (t-t_i)^m \mathcal{P}_m\right) \\
&=& \sum_{n=0}^{\infty} (t-t_i)^n \left(\sum_{k=0}^{n} f_k\mathcal{P}_{n-k}\right)
\end{eqnarray}
for all $t$ in a neighbourhood of $t_i$ where $\mathcal{P}_n$ and $f_n$ are defined in lemma \ref{ct} and lemma \ref{full holo} respectively. Define $\l_n= \sum_{k=0}^{n} f_k\mathcal{P}_{n-k}$. Since $f_k \in H^2(D)$ is bounded and $\mathcal{P}_k: H^2(D) \to H^2(D)$ it follows that $\l_n:H^2(D) \to H^2(D)$. Moreover by (\ref{bf}), the triangle inequality and the bounds on $\norm{f_k}_{\infty}$ and $\norm{\mathcal{P}_k}$ it follows that $\norm{\l_n} \leq C_5^n$ for some constant $C_5$. For (b) we can instead apply lemma \ref{full holo}(b) and employ an analogous argument.
\end{proof}

Let $\t_0 \in (0,1)^{4|\I|}\cap \Omega$, so that $\A_{\t_0}$ is an irreducible set of positive invertible matrices and suppose $s_0= \dim \A_{\t_0} \in (0,1)\cup(1,2)$. Using the analyticity of the family $\{\l_{s,\t}\}$ in a complex neighbourhood of $(s_0,\t_0)$ (corollary \ref{full holo}) and the simplicity of $\lambda_1(s_0,\t_0)$ (lemma \ref{ev}) we can use proposition \ref{kato} to deduce that $\lambda_1(s,\t)$ is simple in a complex neighbourhood of $(s_0, \t_0)$.

\begin{lma} \label{simple2}
Suppose $\t_0 \in \Omega \cap (0,1)^{4|\I|}$ such that $s_0= \dim \A_{\t_0} \in (0,1) \cup (1,2)$. For all $\epsilon>0$ there exists a complex neighbourhood $U_{\epsilon}$ of $(s_0, \t_0)$ such that for all $(s,\t) \in U_{\epsilon}$,
$$\norm{\l_{s,\t}-\l_{s_0,\t_0}} < \epsilon.$$
In particular there exists a complex neighbourhood $\Upsilon$ of $(s_0,\t_0)$ such that $\lambda_1(s,\t)$ is a simple eigenvalue of $\l_{s,\t}$ for all $(s,\t) \in \Upsilon$.
\end{lma}

\begin{proof}
In view of proposition \ref{kato} and lemma \ref{ev} it is sufficient to prove the first part of the lemma.

Let $f \in H^2(D)$. Then
\begin{eqnarray*}
\norm{\l_{s,\t}f-\l_{s_0,\t_0}f}_{H^2} &\leq& \norm{\l_{s,\t}f-\l_{s_0,\t_0}f}_{\infty}\\
&\leq& \sum_{i \in \I} \sum_{n \in \N} \bignorm{\alpha_n(f)2^n\left(\psi_{A_{\t,s}^{(i)}}(\phi_{A_{\t}^{(i)}}-\frac{1}{2})^n-\psi_{A_{\t_0,s_0}^{(i)}}(\phi_{A_{\t_0}^{(i)}}-\frac{1}{2})^n\right)}_{\infty} \\
&\leq& \norm{f}_{H^2}\sum_{i \in \I} \left(\sum_{n \in \N} 2^{2n} \left(\norm{\psi_{A_{\t,s}^{(i)}}-\psi_{A_{\t_0,s_0}^{(i)}}}_{\infty}\bignorm{\phi_{A_{\t}^{(i)}}-\frac{1}{2}}_{\infty}^n+ \right.\right. \\
& & \left.\left. \norm{\psi_{A_{\t_0,s_0}^{(i)}}}_{\infty}\bignorm{(\phi_{A_{\t}^{(i)}}-\frac{1}{2})^n-(\phi_{A_{\t_0}^{(i)}}-\frac{1}{2})^n}_{\infty}\right)^2\right)^{\frac{1}{2}}
\end{eqnarray*}
where the final line follows by the Cauchy-Schwarz and triangle inequalities.

It is sufficient to show that there exists some $R \in (0,1)$ such that for all $\epsilon>0$ we can find some neighbourhood $U_{\epsilon}$ of $(s_0,\t_0)$ such that for all $(s,\t) \in U_{\epsilon}$ and any $i \in \I$,
\begin{eqnarray}
\norm{\psi_{A_{\t,s}^{(i)}}-\psi_{A_{\t_0,s_0}^{(i)}}}_{\infty}\bignorm{\phi_{A_{\t}^{(i)}}-\frac{1}{2}}_{\infty}^n+  \norm{\psi_{A_{\t_0,s_0}^{(i)}}}_{\infty}\bignorm{(\phi_{A_{\t}^{(i)}}-\frac{1}{2})^n-(\phi_{A_{\t_0}^{(i)}}-\frac{1}{2})^n}_{\infty} \leq \epsilon \left(\frac{R}{2}\right)^n. \label{key}
\end{eqnarray}

It follows from our assumptions that there exists $r \in (0,1)$ such that for all $i \in \I$, $\norm{\phi_{A_{\t_0}^{(i)}}-\frac{1}{2}}_{\infty} \leq \frac{r}{2}$. Fix $r<R'<R<1$ and note that for all $\t$ in a neighbourhood of $\t_0$, $\norm{\phi_{A_{\t}^{(i)}}-\frac{1}{2}}_{\infty} \leq \frac{R'}{2}$. Therefore it is easy to see that there exists a neighbourhood $U^1_{\epsilon}$ of $(s_0,\t_0)$ such that for all $(s,\t) \in U^1_{\epsilon}$,
$$\norm{\psi_{A_{\t,s}^{(i)}}-\psi_{A_{\t_0,s_0}^{(i)}}}_{\infty}\bignorm{\phi_{A_{\t}^{(i)}}-\frac{1}{2}}_{\infty}^n \leq \frac{\epsilon}{2}\left(\frac{R}{2}\right)^n.$$
Next we consider the second term in (\ref{key}). Consider 
\begin{eqnarray}
\bignorm{(\phi_{A_{\t}^{(i)}}-\frac{1}{2})^n-(\phi_{A_{\t_0}^{(i)}}-\frac{1}{2})^n}_{\infty}=\left(\frac{R}{2}\right)^n \bignorm{\left(\frac{\phi_{A_{\t}^{(i)}}-\frac{1}{2}}{\frac{R}{2}}\right)^n-\left(\frac{\phi_{A_{\t_0}^{(i)}}-\frac{1}{2}}{\frac{R}{2}}\right)^n}_{\infty}.\label{normalised}
\end{eqnarray}
By putting $R''=\frac{R'}{R} \in (0,1)$ we have
$$\bignorm{\frac{\phi_{A_{\t}^{(i)}}-\frac{1}{2}}{\frac{R}{2}}}_{\infty}, \bignorm{\frac{\phi_{A_{\t_0}^{(i)}}-\frac{1}{2}}{\frac{R}{2}}}_{\infty}<R''.$$
Let $C_1$ be any constant such that $\norm{\psi_{A_{\t_0,s_0}^{(i)}}}_{\infty} \leq C_1$ and $C_2$ be any upper bound on the sequence $N(R'')^N$. Let $U^2_{\epsilon}$ be a neighbourhood of $(s_0,\t_0)$ such that for all $(s,\t) \in U^2_{\epsilon}$,
$$\bignorm{\frac{\phi_{A_{\t}^{(i)}}-\frac{1}{2}}{\frac{R}{2}}-\frac{\phi_{A_{\t_0}^{(i)}}-\frac{1}{2}}{\frac{R}{2}}}_{\infty} \leq \frac{\epsilon}{2C_1C_2}.$$
Then by (\ref{normalised}) and the identity
$$x^n-y^n=(x-y)(x^{n-1}+x^{n-2}y+ \cdots + xy^{n-2}+y^{n-1})$$
we can deduce that
$$
\norm{\psi_{A_{\t_0,s_0}^{(i)}}}_{\infty}\bignorm{(\phi_{A_{\t}^{(i)}}-\frac{1}{2})^n-(\phi_{A_{\t_0}^{(i)}}-\frac{1}{2})^n}_{\infty} 
\leq C_1 \left(\frac{R}{2}\right)^n\frac{\epsilon}{2C_1C_2}n(R'')^n\leq \frac{\epsilon}{2}\left(\frac{R}{2}\right)^n
$$
for all $(s,\t) \in U_{\epsilon}^2$. Taking $U_{\epsilon}=U_{\epsilon}^1 \cap U_{\epsilon}^2$ completes the proof.
\end{proof}

\subsection{Proofs of main results} \label{proof}

By corollary \ref{holo op} and lemma \ref{simple2} we can apply proposition \ref{apt} to deduce that $\lambda_1(s,\t)$ is separately analytic in $s$ and in each $t_i$. By lemma \ref{ev}(b) this immediately implies the analyticity of $P(s,\t)$, but in order to deduce the analyticity of $\dim \A_{\t}$ we need to invoke the implicit function theorem. To this end, we will require the following result.

\begin{lma}
Fix some $\t_0 \in (0,1)^{4|\I|} \cap \Omega$ so that $\A_{\t_0}$ is an irreducible set of positive invertible matrices which are contracting with respect to some norm on $\mathbb{R}^2$, and suppose $s_0= \dim \A_{\t_0} \in (0,1) \cup (1,2) $. Then
$$\frac{\partial }{\partial s} \lambda_1(s,\t) \biggr|_{(s,\t)=(s_{0},\t_0)} \neq 0.$$
 \label{lyap}
\end{lma}
\begin{proof}
By Lemma \ref{ev}(b) it is sufficient to show that
$$\frac{\partial }{\partial s} P(s,\t) \biggr|_{(s,\t)=(s_{0},\t_0)} <0.$$
In view of the definition of $P(s,\t)$ it is sufficient to show that there exist $C>0$ and $\theta \in (0,1)$ such that for all $n \geq 1$ and all sufficiently small $\varepsilon>0$
\[\sum_{A \in \mathcal{A}_{\t_0}^n} \phi^{s_0+\varepsilon}(A) \leq C\theta^{n\varepsilon} \sum_{A \in \mathcal{A}_{\t_0}^n} \phi^{s_0}(A).\]
Let $|||\cdot|||$ be a norm on $\R^2$ such that $|||A|||<1$ for all $A \in \mathcal{A}_{\t_0}$, and let $C>0$ such that $C^{-1}\|B\| \leq |||B|||\leq C\|B\|$ for all $2\times 2$ real matrices $B$. Denote $\theta=\max_{A \in \mathcal{A}_{\t_0}} |||A|||<1$. If $s_0+\varepsilon \leq 1$ we note that
\begin{align*}\sum_{A \in \mathcal{A}_{\t_0}^n} \phi^{s_0+\varepsilon} &= \sum_{A \in \mathcal{A}_{\t_0}^n} \|A\|^{s_0+\varepsilon}  \leq C^{s_0+\varepsilon} \sum_{A \in \mathcal{A}_{\t_0}^n} |||A|||^{s_0+\varepsilon}\\
&\leq C^{s_0+\varepsilon} \left(\max_{A \in \mathcal{A}_{\t_0}^n} |||A|||\right)^{\varepsilon} \sum_{A \in \mathcal{A}_{\t_0}^n} |||A|||^{s_0} \\
& \leq C^{s_0+\varepsilon} \left(\max_{A \in \mathcal{A}_{\t_0}} |||A|||\right)^{n\varepsilon} \sum_{A \in \mathcal{A}_{\t_0}^n} |||A|||^{s_0} \\
& \leq C^{2s_0+\varepsilon} \theta^{n\varepsilon} \sum_{A \in \mathcal{A}_{\t_0}^n} \|A\|^{s_0}  = C^{2s_0+\varepsilon} \theta^{n\varepsilon} \sum_{A \in \mathcal{A}_{\t_0}^n} \phi^{s_0}(A),\end{align*}
say, and the result follows. The case where $s_0>1$ is similar.
\end{proof}

\noindent \emph{Proof of Theorem \ref{main}.} Fix $\t_0 \in (0,1)^{4|\I|} $ such that $\A_{\t_0}$ is an irreducible set of invertible matrices and assume $s_0= \dim \A_{\t_0} \in (0,1) \cup (1,2)$. In particular $\t_0 \in \Omega$. Let $U_1, \ldots, U_{4|\I|}, V$ be the neighbourhoods from corollary \ref{holo op}, and by lemma \ref{simple2} we can assume that $\lambda_1(s,\t)$ is simple for all $(s,\t) \in U_1 \times \cdots \times U_{4|\I|} \times V$. By corollary \ref{holo op}, $\{\l_{s,\t}\}$ is analytic in $s$ in a neighbourhood of $s_0$ while $t_i\in U_i$ are fixed for all $i$. Since $\lambda_1(s_0,\t)$ is simple for all $\t\in U_1 \times \cdots \times U_{4|\I|}$, we can invoke the analytic perturbation theorem (proposition \ref{apt}) to deduce that $\lambda_1(s,\t)$ is analytic in $s$ in some neighbourhood of $s_0$ while $\t \in U_1 \times \cdots U_{4|\I|}$ is fixed. By using the analogous argument for each $t_i$ and applying Hartogs's theorem, we obtain that $\lambda_1(s,\t)$ is jointly analytic in $(s,\t)$ in a neighbourhood $\Upsilon'$ of $(s_0, \t_0)$. Therefore, $(s,\t) \to  P(s,\t)$ is real analytic in $\Upsilon' \cap \R^{4|\I|+1}$.

Define the analytic map $F(s,\t):=P(s,\t)-1$ which satisfies $F(s_0, \t_0)=0$ by assumption. Observe that $\frac{\partial F}{\partial s}(s_0, \t_0) \neq 0$ by Lemma \ref{lyap}. Therefore by the implicit function theorem there exists an analytic function $\delta: B(\t_0, \epsilon) \to B(s_0, \epsilon')$ such that $F(\t, \delta(\t))=0$ for all $\t \in B(\t_0, \epsilon)$. In particular by the uniqueness of the root of the pressure, $\delta(\t)= \dim \A_{\t}$ which completes the proof.   \qed

\vspace{2mm}

\noindent \textit{Proof of Corollary \ref{bhr cor}.}  
Suppose $\A_{\t_0}$ is a set of irreducible matrices that strictly preserve a common cone $\mathcal{C}$. Then it is easy to see that for $\t$ in an real open neighbourhood of $\t_0$, $\A_{\t}$ is also a set of irreducible matrices which also strictly preserve $\mathcal{C}$. Note that $\A_{\t}$ is in fact necessarily strongly irreducible since it preserves $\mathcal{C}$. It is also easy to see that the set of normalised matrices $\tilde{\A_{\t}}$ generate a non-compact subgroup of $\mathcal{GL}_2(\R)$ , since if $A \in \A_{\t}$ then $\lambda_1(A) > \sqrt{|\det(A)|}$ by the Perron-Frobenius theorem for positive matrices and therefore $\norm{\frac{1}{\sqrt{|\det(A^n)|}}A^n} \to \infty$ as $n \to \infty$. 

Also, since $\Phi_{\t_0}$ satisfies the strong separation condition, $\Phi_{\t}$ satisfies the strong separation condition (which implies the strong open set condition) for $\t$ in a real open neighbourhood of $\t_0$. 

Combining all of this together, we see that $\Phi_{\t}$ satisfies the hypothesis of theorem \ref{bhr} for any $\t$ in an open neighbourhood of $\t_0$, and therefore $\dim \A_{\t}=\hd F_{\t}=\bd F_{\t}$ for $\t$ in an open neighbourhood of $\t_0$. Applying theorem \ref{main} completes the proof. \qed

\vspace{2mm}

\noindent \textit{Proof of Corollary \ref{cor2}.} As in the proof of corollary \ref{bhr cor} we can deduce that $\tilde{\A_{\t}}$ generate a non-compact strongly irreducible subgroup of $\mathcal{GL}_2(\R)$ for all $\t$ in some open neighbourhood of $\t_0$. By assumption, $\Phi_{\t}$ satisfies the strong open set condition for $\t$ close to $\t_0$ and therefore for all $\t$ in an open neighbbourhood of $\t_0$, $\Phi_{\t}$ satisfies the hypothesis of theorem \ref{bhr}. Thus $\hd F_{\t}= \bd F_{\t}=\dim \A_{\t}$ and by applying theorem \ref{main} the proof is complete. \qed

\end{document}